\documentclass[a4paper,10pt]{amsart}
\usepackage[mathscr]{eucal}
\usepackage{amssymb}
\usepackage{amsmath,amsthm}
\usepackage{latexsym}
\usepackage{parskip}
\usepackage{hyperref}
\newtheorem{theorem}{Theorem}[section]

\newtheorem{corollary}[theorem]{Corollary}

\theoremstyle{definition}

\newtheorem{example}[theorem]{Example}

\theoremstyle{remark}
\newtheorem{remark}[theorem]{Remark}

\newcommand{\I}{\mathbf{1}}
\numberwithin{equation}{section}

\newcommand{\Cov}{\text{\rm Cov}}

\subjclass[2000]{Primary primary classifications
  Secondary secondary classifications}

\begin{document}

\title[Long Memory Volterra Differential Equation]
{Long Memory in a Linear Stochastic Volterra Differential Equation}
\author{John A. D. Appleby}
\address{School of Mathematical Sciences, Dublin City University, Dublin 9,
Ireland} \email{john.appleby@dcu.ie}
\urladdr{http://webpages.dcu.ie/\textasciitilde applebyj}
\author{Katja Krol}
\address{Humboldt--Universit\"at zu Berlin, Institut f\"ur Mathematik, Unter den
Linden 6, 10099 Berlin, Germany}
\email{krol@mathematik.hu-berlin.de}

\thanks{JA is partly supported by Science Foundation Ireland under the Mathematics Initiative 2007
grant 07/MI/008 ``Edgeworth Centre for Financial Mathematics''.}
\thanks{KK is supported by the Deutsche Telekom Stiftung}.\subjclass{Primary: 34K20, 34K25, 39A11, 45D05, 60G10; Secondary:  60G15, 60H10, 60H20, 60K05}

\date{\today}

\keywords{Volterra integrodifferential equations, Volterra
difference equations, It\^{o}--Volterra integrodifferential
equations, differential resolvent, asymptotic stability, infinite
moment, stationary solutions, long memory, long--range dependency,
renewal sequences with infinite mean, regular variation,
subexponential.}

\begin{abstract}
In this paper we consider a linear stochastic Volterra equation
which has a stationary solution. We show that when the kernel of the
fundamental solution is  regularly varying at infinity with a
log-convex tail integral, then the autocovariance function of the
stationary solution is also regularly varying at infinity and its
exact pointwise rate of decay can be determined. Moreover, it can be
shown  that this stationary process has either long memory in the
sense that the autocovariance function is not integrable over the
reals or is subexponential. Under certain conditions  upon the
kernel, even arbitrarily slow decay rates of the autocovariance
function can be achieved. Analogous results are obtained for the
corresponding discrete equation.
\end{abstract}

\maketitle


\section{Introduction}

In recent years, much attention in quantitative finance has centred
on the question of whether financial markets are efficient, and
whether there is a significant impact of past events on the current
state of the system, see e.g. Cont \cite{Cont}. A mathematical way
in which this phenomenon can be captured is through the theory of
\emph{long range dependence}, or \emph{long memory}. For continuous
time processes, this is measured by the autocovariance function of a
stationary process being non--integrable and polynomially decaying,
so it must decay more slowly than exponentially. Processes with long
memory also arise in other areas of science such as data network
traffic or hydrology see e.g.~Doukhan et al.~\cite{Doukhan}.

In this paper, we describe a class of processes, both in discrete
and continuous time which exhibit long range dependence through
non--exponential convergence of their autocovariance functions. In
the continuous case, these are solutions of scalar affine stochastic
Volterra equations of the form
\begin{align}\label{eq.stoch1}
 dX(t)= \left(aX(t) + \int_{0}^t k(t-s)X(s)\,ds\right)dt + \sigma\,dB(t)
 \quad\text{for }t\ge 0,
\end{align}
where $B$ is standard Brownian motion and $k$ is an integrable
function.
Applications of such equations stochastic Volterra equations arise
in physics and mathematical finance. In physics, for example, the
behaviour of viscoelastic materials under external stochastic loads
has  been analysed using It\^{o}--Volterra equations (cf.,
e.g.~Drozdov and Kolmanovski\u{\i} \cite{Kolm}). In financial mathematics, the presence of
inefficiency in real markets can be modelled by using stochastic
functional differential equations. Anh et al.~\cite{Anh1, Anh2} have
posited models for the evolution of asset returns using stochastic
Volterra equations with infinite memory.

For affine stochastic functional differential equations with bounded
delay, it has been shown that stationary solutions always have
exponentially fading autocovariance function, see e.g.~Gushchin and
K{\"u}chler \cite{Gush_Kue1}, Riedle \cite{Riedle}. This is a
consequence of the fact that, if an autonomous linear differential
equation with finite delay is stable, then its resolvent converges
to zero at an exponentially fast rate, see Hale and
Lunel~\cite{HaleLun:1993a}.

In order to obtain polynomial convergence results for linear
autonomous Volterra equations, it is necessary to consider kernels
$k$ which decay non--exponentially, both for deterministic and
stochastic equations. While a substantial literature exists in the
deterministic case (see e.g.,
~\cite{SheaWaing:1975,GelfandRaikShil:1964,JordanWheeler:1980,jadwr:2002b,jaigdwr:2004b,
App_exact_discrete}) only a few results for non--exponential
convergence phenomena of linear stochastic autonomous Volterra
equations exist, and those that do concern the asymptotic stability
of point equilibria. Examples of such papers include Appleby
\cite{Appleby1, Appleby2} for pointwise convergence rates, Appleby
and Riedle \cite{appleby_riedle} for convergence rates in weighted
$L^p$--spaces, and Mao and Riedle for mean square convergence rates
\cite{MaoRiedle:06}. In particular, polynomial convergence rates of
the autocovariance function of \eqref{eq.stoch1} have not been
recorded.

In this paper, we examine the asymptotic behaviour of the
autocovariance function of asymptotically stationary solutions of
\eqref{eq.stoch1}. To do this, our first class of results concerns
the exact rate of convergence to zero of the solution of the
differential resolvent associated with \eqref{eq.stoch1}, namely
\begin{equation}    \label{eq.fund1}
r'(t) = ar(t)+ \int_{0}^t k(t-s)r(s)\,ds \quad\text{for } t \ge
0,\quad
 r(0)=1.
\end{equation}
We consider first equations for which the kernel $k$ is positive and
integrable with infinite first moment. In this case it is only known
to date that the resolvent $r$ converges to zero and is not
integrable.

%

In this paper we first show that if the kernel $k$ additionally
satisfies $a+\int_0^\infty k(s)\,ds=0$ and the tail integral
$\lambda(t):=\int_t^\infty k(s)\,ds$ is a log--convex regularly
varying function with index $\alpha$, then the solution $r$ is
decays at a hyperbolic rate, according to
\begin{equation} \label{r_lim1}
\lim_{t\to\infty} r(t)t^{1-\alpha}L(t)=\frac{\sin \alpha \pi}{\pi},
\end{equation}
where $L$ is a slowly varying function related to $k$. Corresponding
asymptotic results are established in discrete time. The discrete
analogue of equation \eqref{eq.fund1} with positive summable kernel
of infinite moment corresponds to the renewal sequence  of a
null--recurrent Markov chain \cite{Giacomin}, and under similar
additional assumptions on the kernel, the hyperbolic decay of the
sequence relies upon  well--known results by Garsia and Lamperti
\cite{Garsia} and Isaac \cite{Isaac}.

Our second class of results in this paper employ the convergence
rate of the resolvent $r$ to investigate the long memory properties
of the solution of the It\^{o}--Volterra differential equation
\eqref{eq.stoch1} and its discrete analogue.
It turns out, that under  the same conditions on the kernel $k$, the
equation \eqref{eq.stoch1} possesses an asymptotically stationary
solution for $0<\alpha<1/2$. There also exists  a limiting equation
which is stationary and its autocovariance function $c$ obeys
\begin{equation} \label{eq:asymptotocs_c1}
\lim_{t\to\infty} c(t) L^2(t)t^{1-2\alpha} =
\sigma^2\frac{\Gamma(1-2\alpha)\Gamma(\alpha)}{\Gamma(1-\alpha)}\cdot
\frac{ \sin^2(\pi \alpha)}{\pi^2}.\end{equation} Moreover, because
$c$ is non--integrable, the process has long memory. Again,
corresponding results hold in discrete time.

If $\alpha>1/2$, no stationary solutions exist and the case
$\alpha=1/2$ turns out to be critical. In this situation, we give
necessary and sufficient conditions for the existence of a
stationary solution and show not only that its autocovariance
function has long memory, but  that it can also decay at an
arbitrarily slow rate in the class of slowly varying functions.

In order to give a complete characterization the asymptotic
behaviour of the autocovariance function of \eqref{eq.stoch1}, we also
treat the cases $a+\int_0^\infty k(s)\,ds<0$  and  $a+\int_0^\infty
k(s)\,ds>0$. While in the latter case no stationary solution exists,
we show in the first case, that under weaker assumptions on the
kernel $k$, the autocovariance function of the stationary solution
is integrable. Nevertheless, its decay is very slow: the rate of
convergence to zero is the same as the decay rate of $k$, that is
hyperbolic.

Although we have mentioned discrete results only briefly in this
introduction, there are many reasons to formulate the models
\eqref{eq.fund1} and \eqref{eq.stoch1} in discrete time. When
modelling dynamic real--world phenomena, it is desirable that
properties formulated in discrete or continuous time should be
consistent. In this paper, our results demonstrate that the long or
subexponential memory are general properties of the Volterra model
and do not depend on the continuity assumption. Secondly, by
applying for example a constant step size Euler--Maruyama scheme to
the continuous equation \eqref{eq.stoch1}, we obtain consistent
estimates of the decay rate of the autocovariance function. These
decay estimates stabilise appropriately to those obtained in the
continuous case in the limit as the step size tends to zero.

\section{Discrete and continuous stochastic Volterra--Equations}
\subsection{Mathematical Preliminaries} \label{prelim}
We denote the spaces of real--valued  continuous functions by $C([0,\infty);\mathbb{R})$. Let $L^p([0,\infty);\mathbb{R})$ ($\ell^p$), $p\ge 1$, denote the space of real--valued  measurable functions $f$ (sequences $(f_n)_{n\in \mathbb{N}}$) satisfying
\[ \int_0^\infty |f(t)|^p\, dt < \infty \qquad \bigl( \sum_{n=0}^\infty |f_n|^p<\infty\bigr).\]
We write $f\sim g$ for $x\to x_0\in \mathbb{ R}\cup \{\pm \infty\}$ if $ \lim_{x\to x_0} f(x)/g(x)=1.$\\
 A function $L\,:\,[0, \infty)\to (0, \infty)$ is {\em slowly varying at infinity} if for all $x>0,$
 \begin{equation}
 \lim_{t\to \infty} \frac{L(xt)}{L(t)}=1.
 \end{equation}
 A function $f$ {\em varies regularly with index $\alpha\in \mathbb{R}$}, $f\in \mbox{RV}_\infty(\alpha)$,  if it is of the form
 \begin{equation}
 f(t)=t^\alpha L(t)\end{equation}
 with $L$ slowly varying, see e.g.~Feller \cite{Feller:1}, Ch.\ VIII.8. \\
The definition of a regularly varying sequence is a counterpart of the continuous definition \cite{Bojanic}: a sequence of positive numbers $(c_n)_{n\in \mathbb{N}}$ is said to be {\em  regularly varying of index $\rho\in \mathbb{R}$} ($c$ is {\em slowly varying} if $\rho=0$), if \[\lim_{n\to\infty} \frac{c_{[\lambda n]}}{c_n}=\lambda^\rho, \quad \mbox{for every }\lambda>0,\]
where $[x]$ denotes the integer part of $x\in \mathbb{R}_+$. A regularly varying sequence is embeddable as the integer
values of a regularly varying function: the function $c(\cdot)$, defined on $[0, \infty)$ by $c(x):=c_{[x]}$ is  regularly varying of index $\rho$.
\subsection{Continuous--time Gaussian Volterra equations}
We first turn our attention to the deterministic Volterra equation
in $\mathbb{R}$:
\begin{align}    \label{eq.fund}
x'(t) = ax(t)+ \int_{0}^t k(t-s)x(s)\,ds \quad\text{for } t \ge
0,\quad
 x(0)=x_0.
\end{align}
 For any $x_0\in\mathbb{R}$  there is a unique $\mathbb{R}$--valued
function $x$ which satisfies \eqref{eq.fund} on $[0,\infty)$. 
The so--called {\em fundamental solution or resolvent of
\eqref{eq.fund}} is the real--valued function
$r:[0,\infty)\to\mathbb{R}$, which is the unique solution of the equation \eqref{eq.fund1}.

Let $(\Omega, \mathcal{F}, \mathbb{P})$ be a probability space equipped
with a filtration $(\mathcal{F}_{t})_{t \geq 0}$, and let
$B=\{B(t):t \geq 0\}$ be a one--dimensional Brownian motion on
this probability space. We will consider the stochastic
integro--differential equation of the form
\begin{align}\label{eq.stoch}
\begin{split}
 dX(t)&= \left(aX(t) + \int_{0}^t k(t-s)X(s)\,ds\right)dt + \sigma\,dB(t)
 \quad\text{for }t\ge 0,\\
  X(0)&=X_0,
\end{split}
\end{align}
where $k$ is a continuous, integrable real--valued function, and
$\sigma$ is a non--zero real constant. The initial condition $X_0$
is a real--valued, $\mathcal{F}_0$--measurable random variable with
$\mathbb{E}|X_0|^2<\infty$ which is independent of $B$. The
existence and uniqueness of a continuous solution $X$ of
\eqref{eq.stoch} with $X(0)=X_0$ $\mathbb{P}$--a.s.\ is covered in
Berger and Mizel~\cite{BergMiz:2}, for instance. Independently, the
existence and uniqueness of solutions of stochastic functional
equations was established in It\^o and Nisio~\cite{ItoNisio:64} and
Mohammed~\cite{Moh:84}. In fact, $X$ has the variation of constants
representation
\begin{equation} \label{eq.varp}
X(t)=r(t)X_0 + \int_0^t r(t-s)\sigma \,dB(s), \quad t\geq 0.
\end{equation}

We first discuss the existence of asymptotically  stationary solutions of \eqref{eq.stoch}.
It transpires that the critical condition to guarantee stationarity is that the fundamental solution $r$ of
\eqref{eq.fund} is in $L^2([0,\infty);\mathbb{R})$.
\begin{theorem}\label{thm.existence covarfn}
Let $k\in L^1([0,\infty);\mathbb{R})\cap C([0,\infty);\mathbb{R})$. Suppose the fundamental solution $r$ of
\eqref{eq.fund} obeys $r\in L^2([0,\infty);\mathbb{R})$. Let $\sigma\in\mathbb{R}\setminus\{0\}$.
Let $X$ be the solution of \eqref{eq.stoch}. Then for every $t\geq 0$ there exists a real--valued function $c$ such that
\begin{equation} \label{eq.limcovar}
c(t):=\lim_{s\to\infty} \Cov (X(s),X(s+t))=\sigma^2\int_0^\infty r(s)r(s+t)\,ds.
\end{equation}
\end{theorem}
The result follows directly from \eqref{eq.varp}, and the fact that $X_0$ is independent of $B$.

The following theorem shows that \eqref{eq.stoch} has a limiting equation which possesses a stationary, rather than
an asymptotically stationary solution. To this end,
let $B_1=\{B_1(t):t\geq 0\}$ and $B_2=\{B_2(t):t\geq 0\}$ be independent standard Brownian motions,
and consider the process $B=\{B(t):t\in \mathbb{R}\}$ defined by
\begin{equation} \label{eq.whollineBM}
B(t)=
\left\{
\begin{array}{cc}
B_1(t), & t>0\\
B_2(-t), &t\leq 0.
\end{array} \right.
\end{equation}
Then $B$ is a standard Brownian motion defined on the whole line.

\begin{theorem}\label{thm.existence covarfn stat}
Let $k\in L^1([0,\infty);\mathbb{R})\cap C([0,\infty);\mathbb{R})$. Suppose the fundamental solution $r$ of
\eqref{eq.fund} obeys $r\in L^2([0,\infty);\mathbb{R})$. Let $\sigma\in\mathbb{R}\setminus\{0\}$.
Let $B=\{B(t):t\in\mathbb{R}\}$ be the standard one--dimensional Brownian motion defined by
\eqref{eq.whollineBM}. Then the unique continuous adapted process which obeys
\begin{align} \label{eq.stochstat}
\begin{split}
dX(t) &= \left(aX(t)+\int_{0}^\infty k(s)X(t-s)\,ds\right)\,dt + \sigma\,dB(t), \quad t>0;\\
X(t) &= \int_{-\infty}^t r(t-s)\sigma \,dB(s), \quad t\leq 0,
\end{split}
\end{align}
is given by
\begin{equation} \label{eq.Xstat}
X(t)= \int_{-\infty}^t r(t-s)\sigma \,dB(s), \quad t\in \mathbb{R}.
\end{equation}
Moreover, $X$ is a stationary zero mean Gaussian process with autocovariance function given by
\begin{equation} \label{eq.statcovar}
c(t)= \Cov (X(s),X(s+t))=\sigma^2\int_0^\infty r(s)r(s+t)\,ds.
\end{equation}
\end{theorem}
It is clear that if $r$ is in $L^2([0,\infty);\mathbb{R})$ that $X$ defined by \eqref{eq.Xstat} is a stationary zero mean Gaussian process
with autocovariance function given by \eqref{eq.statcovar}. To show that $X$ satisfies \eqref{eq.stochstat} requires more work,
and a proof is given in Section \ref{proof1}.

Theorem~\ref{thm.existence covarfn stat} provides direction for the investigations in this paper.
It is readily seen that $r\in L^1([0,\infty);\mathbb{R})$ implies $c\in L^1([0,\infty);\mathbb{R})$.
Therefore in order to possess long memory 
but still to have stationary solutions, we need to consider conditions on the kernel $k$ in \eqref{eq.fund}
such that the fundamental solution $r$ of \eqref{eq.fund} obeys $r\in L^2([0,\infty);\mathbb{R})$
but $r\not\in L^1([0,\infty);\mathbb{R})$.

Section \ref{section:continuous} gives an example of how this can be achieved.
The crucial hypotheses on $k$ is that it is regularly varying and its tail integral is log--convex:
this enables us to prove that $r$ is regularly varying and to determine the exact rate of decay of $r$.
We then show how the asymptotic behaviour of $c$ can be inferred from $r$ when
$r$ is regularly varying in such a way that $r\in L^2([0,\infty);\mathbb{R})$
but $r\not\in L^1([0,\infty);\mathbb{R})$. The results enable us to determine the exact rate of
decay of the autocovariance function $c$ in terms of the rate of decay of  $k$.
\subsection{Discrete--time Volterra equations}
Let $(\Omega, \mathcal{F}, \mathbb{P})$ be a probability space equipped
with a filtration $(\mathcal{F}_{n})_{n\in \mathbb{N}}$. We consider the discrete version of \eqref{eq.stoch}:
\begin{align}
\label{discr_stoch}
\begin{split}
X_{n+1}-X_{n}&=aX_n +\sum_{j=1}^n k_j X_{n-j} +\xi_{n+1}, \quad n\ge 0,\\
X_0=x_0, \end{split}\end{align}
where $k$ is a positive summable kernel, $a:=-\sum_{j=1}^\infty k_j$ and $\xi=\{\xi_n \,:\, n\in \mathbb{N}\}$ is a sequence of independent,
identically distributed random variables with $\mathbb{E}(\xi_n)=0, \, \mathbb{E}(\xi_n^2)=\sigma^2>0$ for all $n\in \mathbb{N}$.
$x_0$ is an $\mathcal{F}_0$--measurable random variable with $\mathbb{E}(x_0^2)<\infty$ which is independent of $\xi$.
Let $r=\{ r_n:n\in \mathbb{N}\}$ denote the fundamental solution of \eqref{discr_stoch}, i.e., the unique solution of
\begin{equation}
r_{n+1}-r_n=ar_n+\sum_{j=1}^n k_j r_{n-j}, \quad n\ge 1, \quad r_0=1.\label{discrete}\end{equation}
For more information on Volterra difference equations, the reader is referred to the book of Elaydi \cite{elaydi}.
An analogous  result to Theorem \ref{thm.existence covarfn stat} holds for \eqref{discr_stoch}:
\begin{theorem} \label{theorem.discrete}
Suppose that $k\in \ell^1$ and the fundamental solution \eqref{discrete} obeys $r\in \ell^2$.
 Then there is a unique adapted process $X$ which obeys
\begin{align}\label{eq.stochstat_discr}\begin{split}
X_{n+1}-X_{n}&=aX_n +\sum_{j=1}^ \infty k_j X_{n-j} +\xi_{n+1}, \quad n\ge 0;\\
X_n &= \sum_{j=-\infty}^n r_{n-j}\xi_j, \quad n< 0,\end{split}
\end{align}
where $\xi$ is extended to $n\in \mathbb{Z}$ by taking an independent copy $\xi^1$ of $\xi$ (defined on the same probability space) and
setting $\xi_{-n}=\xi^1_n, \, n\in \mathbb{N}$. $X$ is a stationary zero mean process with autocovariance function given by
\begin{equation}\label{cov_discrete}
c(h)= \Cov (X_n,X_{n+h})=\sigma^2\sum_{n=0}^\infty r_n r_{n+h}, \quad h\in\mathbb{N}.
\end{equation}
\end{theorem}
Again, we are able to show that if $(k_n)_{n\in \mathbb{N}}$ is a so--called Kaluza--sequence, then $r$ satisfies $r\in \ell^2$ but $r\not \in \ell^1$ with exact rate of decay specified. From \eqref{cov_discrete} we can deduce the exact asymptotic behaviour of the autocovariance function of the stationary solution.
\section{Long memory in the continuous equation} \label{section:continuous}
\subsection{Asymptotic Behaviour of the Deterministic Resolvent} \label{sec:deterministic}
This section gives the exact rate of decay of the solution of a scalar linear Volterra
differential equation with a non--integrable solution $r$ which nonetheless obeys
$r(t)\to 0$ as $t\to\infty$. 
Suppose that $a+\int_0^\infty k(s)\, ds=0$ and let  $k$ satisfy the following conditions
\begin{enumerate}
\item[(C1)] $k\in L^1([0,\infty);(0,\infty))\cap C([0,\infty);(0,\infty))$,
\item[(C2)]$t\mapsto \log \lambda(t)$ is a convex function, where
\begin{equation}\label{def_lambda}
\lambda(t):=\int_t^\infty k(s)\, ds, \end{equation}
\item[(C3)]
$\lambda(t)=L(t)t^{-\alpha}$ with  $\alpha\in(0,1)$ and a slowly varying at infinity function $L$. 
\end{enumerate}
\begin{remark}
The last two conditions are satisfied, if $k$ is a completely monotone function such that $k\in \mbox{RV}_\infty(-1-\alpha)$. Condition (C2) is equivalent to
\begin{itemize}
 \item[(C2*)] $\displaystyle
\frac{\lambda(t)}{\lambda(t+T)} \quad \mbox{is non--increasing in $t$ for all $T>0$}$.
\end{itemize} Proofs can be found in Miller \cite{Miller2}.
\end{remark}
Condition (C1) implies existence of a unique continuous function $r$ which is a
solution of the integro--differential equation \eqref{eq.fund1}.
In particular, it follows from  (C3) that $k$ obeys
\begin{equation}
\label{eq:eqk12} \int_0^\infty sk(s)\,ds=\infty.
\end{equation}
In this case it is only known that the differential resolvent $r$ satisfies
\begin{equation} \label{eq:eqrnotl1}
\lim_{t\to\infty} r(t)=0, \quad r\not\in
L^1((0,\infty);(0,\infty)).
\end{equation}
\begin{theorem} \label{thm:explicitasymptotics}
Suppose that $k$ obeys (C1)--(C3). 
 If $r$ is
the unique continuous solution of \eqref{eq.fund1}, then
\begin{equation} \label{r_lim}
\lim_{t\to\infty} r(t)t^{1-\alpha}L(t)=\frac{\sin \alpha \pi}{\pi}.
\end{equation}
\end{theorem}
Hence for $\alpha\in (0,1/2)$ we have $r\in L^2([0,\infty);(0,\infty))$ but $r\not \in L^1([0,\infty);(0,\infty))$ due to $r\in \text{RV}_\infty(\mu)$ for
$\mu=\alpha-1\in(-1,-1/2)$.
\begin{proof}
We note that $\lambda\in C^1((0,\infty);(0,\infty))$. Evidently $\lambda$ is positive, non--increasing, satisfies
$\lambda(t)\to0$ as $t\to\infty$. Though, by virtue of
(C3) this happens so slowly that $\lambda\not\in
L^1([0,\infty);\mathbb{R})$.\\
Since $r\in C^1((0,\infty);(0,\infty))$, we can also introduce the function
$\rho=-r'$.

By differentiation of the function $f(t)=r(t)+\int_0^t
\lambda(t-s)r(s)\,ds$, and using \eqref{eq.fund1}, we see that
$f'(t)=0$. Since $f(0)=r(0)=1$, we have
\begin{equation}\label{eq:eqrlambda}
r(t)+\int_0^t \lambda(t-s)r(s)\,ds=1, \quad t\geq 0.
\end{equation}
Therefore,
\begin{align*}
\rho(t)&=-r'(t)=\frac{d}{dt}\left(-1+\int_0^t
\lambda(s)r(t-s)\,ds\right)\\
&=\int_0^t \lambda(s)r'(t-s)\,ds +
\lambda(t)r(0)=\lambda(t)-\int_0^t \lambda(t-s)\rho(s)\,ds.
\end{align*}
Hence $\rho$ is the integral resolvent of $\lambda$.
Now by (C2) and  and Theorem 1.2 in \cite{Levin}, it follows that
\begin{equation} 0\le \rho(t)\le \lambda(t)\quad \mbox{for all $t>0$}, \quad \int_0^\infty \rho(t)d t =1,\label{eq:rho}\end{equation}
particularly implying  $0\le r(t)\le 1$ for all $t\ge 0$.
Since $\lambda(t)\ge0$, we may define a measure $\Lambda$ by
$\Lambda([0,t])=\int_0^t \lambda(s)\,ds$. Then
\[
\omega_{\Lambda}(z):=\int_0^\infty
e^{-zt}\Lambda(dt)=\hat{\lambda}(z).
\]
By (C3), it follows that $\Lambda\in \text{RV}_\infty(1-\alpha)$, so as $1-\alpha>0$,
we can apply Theorem XIII.5.1 in~\cite{Feller:1} to get
\begin{equation}\label{eq:eqlambdaasy}
\hat{\lambda}(\tau)=\omega_{\Lambda}(\tau)\sim
\Gamma(-\alpha+2)\Lambda(1/\tau), \quad\text{as $\tau\to0$}.
\end{equation}
Next, as $r(t)>0$ for all $t\geq 0$, we may define the measure $U$
by $U([0,t])=\int_0^t r(s)\,ds$. Then $u(t):=U'(t)=r(t)$ obeys
$u'(t)=r'(t)=-\rho(t)\leq0$ for all $t\geq0$. Furthermore
\[
\omega_{U}(z):=\int_0^\infty e^{-zt}U(dt)=\int_0^\infty
e^{-zt}r(t)\,dt = \hat{r}(z).
\]
Since $\lambda(t)\to0$ and $r(t)\to0$ as $t\to\infty$, $\hat{\lambda}(z)$ and $\hat{r}(z)$ exist for
$\Re(z)>0$. Therefore, by \eqref{eq:eqrlambda}, we have
\[
\hat{r}(z)+\hat{\lambda}(z)\hat{r}(z)=\frac{1}{z}, \quad \Re(z)>0.
\]
Therefore, for $\tau>0$,
\[
\omega_U(\tau)=\hat{r}(\tau)=\frac{1}{\tau+\tau\hat{\lambda}(\tau)}.
\]
Now, by \eqref{eq:eqlambdaasy}
\[
\tau\hat{\lambda}(\tau)\sim\Gamma(-\alpha+2)\tau\Lambda(1/\tau),
\quad\text{as $\tau\to0$}.
\]
Because $\Lambda\in \text{RV}_\infty(-\alpha+1)$,
$\Lambda_1(\tau):=\tau\Lambda(1/\tau)$ obeys $\Lambda_1\in
\text{RV}_0(\alpha)$. Since $\alpha\in(0,1)$,
$\tau+\tau\hat{\lambda}(\tau)\sim
\Gamma(2-\alpha)\Lambda_1(\tau)=\Gamma(2-\alpha)\tau\Lambda(1/\tau)$
as $\tau\to0$. Thus
\begin{equation}\label{eq:omegauasy}
\omega_U(\tau)=\frac{1}{\tau+\tau\hat{\lambda}(\tau)}\sim\frac{1}{\Gamma(2-\alpha)\tau\Lambda(1/\tau)}
=\frac{1}{\tau^\alpha}L(1/\tau),\quad \text{as $\tau\to0$},
\end{equation}
where
\[
L(1/\tau)=\frac{1}{\Gamma(2-\alpha)}\frac{\tau^{\alpha-1}}{\Lambda(1/\tau)},
\]
which is a slowly varying function by virtue of the fact that
$\Lambda\in \text{RV}_\infty(-\alpha+1)$. Then, as $U$ has a monotone
derivative $u$, and \eqref{eq:omegauasy} holds, Theorem XIII.5.4
in~\cite{Feller:1} implies that
\[
u(t)\sim \frac{1}{\Gamma(\alpha)}t^{\alpha-1}L(t), \quad\text{as
$t\to\infty$}.
\]
Since $u(t)=r(t)$, by the definition of $L$
\[
r(t)\sim \frac{1}{\Gamma(\alpha)}t^{\alpha-1}
\cdot\frac{1}{\Gamma(2-\alpha)}\frac{t^{-\alpha+1}}{\Lambda(t)}
=\frac{1}{\Gamma(\alpha)\Gamma(2-\alpha)}\frac{1}{\Lambda(t)},
\quad \text{as $t\to\infty$}.
\]
Moreover, we have from Proposition 1.5.8 in \cite{Goldie}, that
\[ \Lambda(t)=\int_0^t s^{-\alpha} L(s)\,ds\sim \frac{1}{1-\alpha}t^{1-\alpha}L(t), \quad \mbox{as } t\to \infty.\]
Hence,
\[
\lim_{t\to\infty} r(t)t^{1-\alpha}L(t)=\frac{1-\alpha}{\Gamma(\alpha)\Gamma(2-\alpha)}=\frac{\sin \alpha \pi}{\pi},
\]
as required.
\end{proof}
For the sake of completeness, we also study the case where $\lambda$, defined as in \eqref{def_lambda}, satisfies $\lambda\in \text{RV}_\infty(-\alpha)$ with  $\alpha>1$. It turns out that in this case  $r$ converges to a positive limit and hence  cannot be asymptotically stable.
\begin{corollary}
Suppose that $k$ satisfies (C1) and (C3) with $\alpha>1$ and that $a+\int_0^\infty k(s)\, ds=0$ holds true. Then,
$ \int_0^\infty sk(s)\, ds<\infty$ and
\begin{equation} \label{r_nonzero_limit}\lim_{t \to \infty} r(t)=\left(1+\int_0^\infty sk(s)\, ds\right)^{-1}.\end{equation}
\end{corollary}
\begin{proof}
Since $\lambda$ is continuous satisfying $\lambda(0)=\int_0^\infty k(s)\, ds<\infty$ and $\lambda\in \text{RV}_\infty(-\alpha)$ with $\alpha>1$, we also have $ \lambda\in L^1([0, \infty);(0, \infty))\cap C([0, \infty); (0, \infty))$. Moreover
\[ \int_0^\infty \lambda(s)\, ds= \int_0^\infty sk(s)\, ds<\infty.\]
Then, Theorem 4.2 in \cite{Ap_Rey_sub} yields \eqref{r_nonzero_limit}.
\end{proof}
%
\subsection{Asymptotic behaviour of the autocovariance function} \label{sec:main}
In this section we state our second main result, Theorem
\ref{theorem.mainresult}, which characterizes completely the
asymptotic rate of convergence of the autocovariance function $c(t)$
of the solution of \eqref{eq.stochstat} for the case when
$a=-\int_0^\infty k(s)\,ds$. In the case where $0<\alpha<1/2$, it
turns out that for the kernels $k$ satisfying (C1)--(C3),
 $c(t)$ resembles the  power law function $t^{2\alpha-1}$ for large values of $t$ and hence exhibits long memory.
 The case where $\alpha=1/2$ is more subtle; indeed, for some such $k$ we have $r\not\in L^2([0,\infty);\mathbb{R})$.
 If $r\in L^2([0,\infty);\mathbb{R})$, it is  still possible to determine the rate of decay of $c$, which continues to exhibit long memory.
 Perhaps the most interesting aspect of this result is that arbitrarily slow rates of decay of $c$
in  $\text{RV}_\infty(0)$ can be obtained.
\begin{theorem} \label{theorem.mainresult}
Suppose that $k$ satisfies (C1)--(C3) with $\alpha\in(0,1/2)$.
Let $r$ be the solution of \eqref{eq.fund1}. 
Let $\sigma\in\mathbb{R}\setminus\{0\}$
and $B=\{B(t):t\in\mathbb{R}\}$ be the standard one--dimensional Brownian motion defined by
\eqref{eq.whollineBM}. Then there is a unique stationary Gaussian process $X$ which obeys \eqref{eq.stochstat}:
\begin{align*}\begin{split}
dX(t) &= \left( a
X(t)+\int_{0}^\infty k(s)X(t-s)\,ds\right)\,dt + \sigma\,dB(t), \quad t>0;\\
X(t) &= \int_{-\infty}^t r(t-s)\sigma \,dB(s), \quad t\leq 0.
\end{split}\end{align*}
The autocovariance function $c(\cdot)=\Cov(X(s),X(s+\cdot))$ satisfies
\begin{equation} \label{eq:asymptotocs_c}
\lim_{t\to\infty} c(t) L^2(t)t^{1-2\alpha}
= \sigma^2\frac{\Gamma(1-2\alpha)\Gamma(\alpha)}{\Gamma(1-\alpha)}\cdot \frac{ \sin^2(\pi \alpha)}{\pi^2}.
\end{equation}
\end{theorem}
\begin{proof}
The proof of the theorem can be found in Section \ref{proof}.
\end{proof}
\begin{example}Let $\alpha\in (0,1/2)$ and \label{ex1}
\begin{equation}
k(t)=\frac{1}{(1+t)^{\alpha+1}}, \quad t\geq0.
\end{equation}
Then, $\lambda(t)=1/(\alpha (1+t)^\alpha)$, $t\ge 0$, and since $L(t)\to 1/\alpha$ as $t\to \infty$, we obtain the following convergence rate of the autocovariance function:
\[ \lim_{t\to \infty}\frac{c(t)}{t^{2\alpha-1}}=\sigma^2\frac{\sin(\alpha \pi) \Gamma(1-2\alpha)}{\pi \Gamma(-\alpha)^2}.\]
\end{example}

We now consider the interesting and critical case where $
\alpha=1/2$. Depending on the properties of the slowly varying
function $L$, both $r\not \in L^2([0, \infty);\mathbb{R})$ as well
as $r\in L^2([0, \infty);\mathbb{R})$ is possible. 
We first determine the rate of convergence
of the autocovariance function.
\begin{theorem}\label{thm:half}
Suppose that $k$ satisfies (C1), (C2) and $k(t)=L(t)t^{-3/2}$, $t\ge
0$, with a slowly varying function $L$. Then, $r\in L^2([0,
\infty);\mathbb{R})$ if and only if
\begin{equation}\label{eq.limit_finite}
\int_1^\infty \frac{1}{tL(t)^2}\, dt<\infty.\end{equation} Moreover,
if \eqref{eq.limit_finite} holds true, then
\[ c(t)\sim\frac{\sigma^2}{\pi^2} \int_t^\infty \frac{1}{sL(s)^2}\, ds, \quad t\to \infty.\]
\end{theorem}
\begin{proof}
Theorem \ref{thm:explicitasymptotics} yields that
\begin{equation}\label{eq.lim_r_12}\lim_{t\to \infty}r(t)t^{1/2}L(t)= \lim_{t\to \infty} r(t)k(t)t^2 =\frac{1}{\pi}.\end{equation}
Since $r$ is continuous on $[0, \infty)$, $r\in L^2([0,
\infty);\mathbb{R})$ if and only if
\[ \int_1^\infty \frac{1}{t^4 k(s)^2}\, dt= \int_1^\infty \frac{1}{tL(t)^2}\, dt<\infty.\]
In this case we denote by
\[
f(t):=\frac{\sigma^2}{\pi^2} \int_t^\infty \frac{1}{s^4 k(s)^2}\,ds,
\; t\ge 0.
\]
The integrand of $f$ is regularly varying with index $-1$. Then, by
Karamata's Theorem (see e.g.~\cite{Goldie}, Theorem 1.5.11) we
obtain
\begin{equation}\label{eq.kar1}\frac{t}{t^4 k^2(t)f(t)}\to 0, \quad \mbox{for $t\to \infty$.}\end{equation}
Moreover, with \eqref{eq.lim_r_12} and \eqref{eq.kar1} it holds that
\begin{equation}\label{eq.kar2}\lim_{t\to \infty}  \frac{tr(t)^2}{f(t)}
=\lim_{t\to \infty}r(t)^2t^4 k(t)^2\lim_{t\to \infty} \frac{t}{t^4
k(t)^2f(t)}= 0.\end{equation} We write
\[ \frac{c(t)}{f(t)}=\frac{\sigma^2}{f(t)}\int_0^t r(s)r(t+s)\, ds+\frac{\sigma^2}{f(t)}\int_t^\infty r(s)r(t+s)\, ds=:I_1(t)+I_2(t), \; t\ge 0.\]
By \eqref{eq:rho}, $r$ is positive and non--increasing, hence we
obtain the following  upper bound for $I_2(t)$:
\begin{equation} \label{lhosp} I_2(t)\le\frac{\sigma^2}{f(t)} \int_t^\infty r(s)^2\, ds, \quad t\ge 0.\end{equation}
The denominator and the numerator in \eqref{lhosp} tend to zero as
$t$ tends to infinity, therefore,  we may apply   L'H{\^o}spital's
rule to obtain
\begin{equation} \label{eq:limit1}
\lim_{t\to \infty} \frac{\sigma^2}{f(t)} \int_t^\infty r(s)^2\,
ds=\lim_{t\to\infty}\pi^2 r(t)^2 t^4 k(t)^2 =1.
\end{equation}
On the other hand,
\begin{equation} \label{lhosp1} I_2(t)\ge\frac{\sigma^2}{f(t)} \int_t^\infty r(s+t)^2\, ds
=\frac{\sigma^2}{f(t)} \int_{2t}^\infty r(s)^2\, ds, \quad t\ge 0.
\end{equation}
By \eqref{eq.kar2} we have
\begin{equation} \label{lhosp2}
\lim_{t\to \infty} \frac{1}{f(t)}\int_t^{2t} r(s)^2\, ds\le
\lim_{t\to \infty}\frac{tr(t)^2}{f(t)}=0.
\end{equation}
Combining \eqref{lhosp}, \eqref{eq:limit1}, \eqref{lhosp1} and
\eqref{lhosp2} we obtain $\lim_{t\to \infty} I_2(t)=1$. The term
$I_1(t)$ vanishes as $t$ tends to infinity: applying Karamata's
theorem to $r\in\mbox{RV}_\infty (-1/2)$ and using \eqref{eq.kar2},
we obtain
\[\lim_{t\to \infty}\frac{ I_1(t)}{\sigma^2}\le \lim_{t\to \infty}\frac{ r(t)}{f(t)}\int_0^t r(s)\, ds
=\lim_{t\to \infty}\frac{\int_0^t r(s)\, ds}{tr(t)}\cdot
\frac{tr(t)^2}{f(t)}=2\lim_{t\to \infty} \frac{tr(t)^2}{f(t)}=0.\]
This completes the proof.
\end{proof}
To see that it is possible to obtain arbitrary rates of decay for
$c$ in the class of slowly varying functions which tend to zero, we
consider such a function $\gamma\in \text{RV}_\infty(0)$. We
demonstrate this claim, under a mild technical assumption on
$\gamma$.
\begin{corollary}
Suppose that $\gamma$ is in $C^1((0,\infty);(0,\infty))$,
$\gamma(t)\to 0$ as $t\to\infty$ and that $-\gamma'\in
\text{RV}_\infty(-1)$. Then $\gamma\in \text{RV}_\infty(0)$ and
there exists $L\in \text{RV}_\infty(0)$ which satisfies
\eqref{eq.limit_finite} and
\begin{equation} \label{eq.constructLgammaasy}
\int_t^\infty \frac{1}{sL^2(s)}\,ds \sim \gamma(t), \quad\text{as
$t\to\infty$}.
\end{equation}
\end{corollary}
\begin{proof}
For any $T>t\geq 0$, we have $\gamma(T)-\gamma(t)=\int_t^T
\gamma'(s)\,ds$. Letting $T\to\infty$, we see that
$\gamma(t)=\int_{t}^\infty -\gamma'(s)\,ds$. $-\gamma'$ is
integrable because $\gamma(t)\to 0$ as $t\to\infty$. The fact that
$-\gamma'\in \text{RV}_\infty(-1)$ and is integrable forces $\gamma$
to be in $\text{RV}_\infty(0)$. Define the function $L:[1,\infty)\to
(0,\infty)$ by
\begin{equation} \label{eq.L2}
L^2(t)=\frac{-1}{t\gamma'(t)}.
\end{equation}
Clearly $L\in \text{RV}_\infty(0)$. Moreover for any $T\geq 1$
\[
\int_1^T \frac{1}{sL^2(s)}\,ds = \int_1^T -\gamma'(s) \,ds
=\gamma(1)-\gamma(T).
\]
Since $\gamma(T)\to 0$ as $T\to\infty$, it follows that $L$ obeys
\eqref{eq.limit_finite}. The asymptotic relation
\eqref{eq.constructLgammaasy} is an obvious consequence of the
construction of $L$.
\end{proof}
\begin{remark}
By applying Theorem~\ref{thm:half}, it can be seen that if $k(t)\sim
t^{-3/2}L(t)$ as $t\to\infty$, where $L$ is given by \eqref{eq.L2},
then $c(t)\sim \sigma^2/\pi^2 \gamma(t)$ as $t\to\infty$. Therefore,
functions $k$ exist such that the rate of convergence of the
autocovariance function is an (essentially) arbitrary function in
$\text{RV}_\infty(0)$. For example, $c(t)$ can decay to zero at a
rate asymptotic to $(\log \log \log \cdots \log t)^{-1}$ as
$t\to\infty$, where there are finitely but arbitrarily many
compositions of logarithms.
\end{remark}


\section{Long memory in the discrete  equation} \label{section_discrete}
In this section we study the discrete counterparts to equations \eqref{eq.fund1} and \eqref{eq.stoch} for some summable kernels $k$ with infinite mean. 
\subsection{Asymptotic Behaviour of the Deterministic Resolvent}
Let us first consider the deterministic equation \eqref{discrete} with
 $a+\sum_{j=1}^\infty k_j=0$.
If $1+a>0$ and $(k_n)_{n\ge 1}$ has infinite mean, the classical renewal theorem yields that $r_n$ converges to zero as $n$ tends to infinity. If $(k_n)_{n\ge 1}$ has a regularly varying tail (Garsia and Lamperti  \cite{Garsia}, Theorem 1.1)  and $(r_n)_{n\in \mathbb{N}}$ is monotone non--increasing (Isaac \cite{Isaac}, Theorem 3.1), the exact convergence rates are also known.

In this section we prove that if the tail $\left(\sum_{j=n}^\infty k_j\right)_{n\ge 1}$ is a so--called Kaluza sequence, which is a discrete analogue of log--convexity, then the sequence $(r_n)_{n\in \mathbb{N}}$ is monotone non--increasing and we can apply the above mentioned theorems.
\begin{theorem} \label{thm_deterministic_discrete}
Let $(k_n)_{n\ge 1}$ be a positive sequence such that $\sum_{j=1}^\infty k_j\le 1$. Moreover, let  $\lambda_n:=\sum_{j=n}^\infty k_j, \: n\ge 1,$ satisfy:
\begin{enumerate}
\item[(C2')] $(\lambda_n)_{n\ge 1}$ is  a Kaluza sequence, that is $\lambda_n^2\le \lambda_{n-1} \lambda_{n+1}$ for all $n\ge 1$,
\item[(C3')] $\lambda_n=L(n)n^{-\alpha}$, where $0<\alpha<1$ and $L(n)$ is a slowly varying sequence.
\end{enumerate}
Then
\[ \lim_{n\to \infty} n^{1-\alpha} L(n) r_n =\frac{\sin \alpha \pi}{\pi}.\]
\end{theorem}
\begin{proof}
Since $(L(n))_{n\in \mathbb{N}}$ is slowly varying, so is the function $x\mapsto L([x])$.
Since $1+a\ge0$, we can apply Theorem 1.1 in \cite{Garsia} to obtain the result for $1/2<\alpha <1$. For $\alpha\le 1/2$ the claim follows from  \cite{Isaac}, Theorem 3.1 if the sequence $(r_n)_{n\ge 0}$ is monotone non--increasing. To show this, we define
\[ a_n:= r_n +\sum_{j=1}^{n-1} r_j \lambda_{n+1-j}, \quad n\ge 0, \]
to obtain
\begin{align*}
a_{ n+1}-a_n&=r_{n+1}-r_n +\sum_{j=0}^{n-1}(\lambda_{n+1-j} -\lambda_{n-j})r_j +r_n \lambda_1\\
&=r_{n+1}-r_n-\sum_{j=0}^{n-1} k_{n-j} r_j +r_n a\\
&=0.\end{align*}
Hence, $(a_n)_{n\ge 0}$ is a constant sequence and equals $a_0=r_0=1$. With $\Delta_n:=-(r_n-r_{n-1})$ we have\begin{align*}
0=a_n-a_{n-1}&= -\Delta_n+\sum_{j=0}^{n-1} r_j \lambda_{n-j} -\sum_{j=0}^{n-2} r_j \lambda_{n-1-j}\\
&=-\Delta_n +\sum_{j=1}^{n-1} \lambda_{n-j}(r_j-r_{j-1})+\lambda_n\\
&= -\Delta_n - \sum_{j=1}^{n-1} \lambda_{n-j} \Delta_j +\lambda_n.
\end{align*}
Therefore, $(\Delta_n)_{n\ge 0}$ satisfies the recurrence relation
\begin{equation}
\Delta_n=\lambda_n -\sum_{j=1}^{n-1} \lambda_{n-j} \Delta_j.\end{equation}
Since $(\lambda_n)_{n\ge 0}$ is a Kaluza sequence, it follows from \cite{Shanbhag} that $\Delta_n$ is non--negative for all $n\ge 0$. Hence, the sequence $(r_n)_{n\ge 0}$ is non--increasing and the claim follows.
\end{proof}
\subsection{Asymptotic behaviour of the autocovariance function}
Now we are able to state the discrete analogue of Theorem \ref{theorem.mainresult}:
\begin{theorem} \label{theorem.mainresult_discrete}
Suppose that $k$  satisfies the assumptions of Theorem \ref{thm_deterministic_discrete} with $\alpha\in(0,1/2)$.
Let $r$ be the solution of \eqref{discrete} and $\xi=\{\xi_n:n\in \mathbb{Z}\}$ be a  sequence of random variables defined as in Theorem \ref{theorem.discrete}. 
 Then there is a unique stationary process $X$ which obeys \eqref{eq.stochstat_discr}:
\begin{align*}
X_{n+1}-X_{n}&=-aX_n +\sum_{j=1}^ \infty k_j X_{n-j} +\xi_{n+1}, \quad n\ge 0;\\
X_n &= \sum_{j=-\infty}^n r_{n-j}\xi_j, \quad n< 0.
\end{align*}
The autocovariance function $c(\cdot)=\Cov (X_n,X_{n+\cdot})$   obeys
\begin{equation} \label{eq:asymptotocs_c_discrete}
\lim_{h\to\infty} c(h) L^2(h)h^{1-2\alpha}
= \sigma^2\frac{\Gamma(1-2\alpha)\Gamma(\alpha)}{\Gamma(1-\alpha)}\cdot \frac{ \sin^2(\pi \alpha)}{\pi^2}.
\end{equation}
\end{theorem}
\begin{proof}
The stationary solution is given by $X(n)=\sum_{j=-\infty}^n r_{n-j}\xi_j, \, n\in \mathbb{Z}$, and its autocovariance function obviously satisfies \eqref{cov_discrete}. Since the sequence $(L(n))_{n\in \mathbb{N}}$ is slowly varying we obtain with Theorem \ref{thm_deterministic_discrete}  for all $\lambda>0$
\begin{align*} \lim_{n\to \infty}\frac{r_{[\lambda n]}}{r_n}&= \lim_{n\to \infty}\frac{L(n)n^{1-\alpha}}{[\lambda n]^{1-\alpha}L([\lambda n])}=
\lim_{n\to \infty}\frac{n^{1-\alpha}}{[\lambda n]^{1-\alpha}}\\&=\lim_{n\to \infty} \left( \lambda +\frac{ [\lambda n               ] -\lambda n}{n}\right)^{\alpha -1}=\lambda ^{\alpha-1}.\end{align*}
Hence the positive sequence $(r_n)_{n\in\mathbb{N}}$ is regularly varying with index $\alpha-1$. Therefore, as mentioned in Section \ref{prelim}, the function $r(x):=r_{[x]}$, $x\ge 0$, is also regularly varying and we may write
\[c(h)=\sigma^2 \int_0^\infty r(x)r(x+h), \quad h\in \mathbb{N}.\]
With Theorem \ref{theorem.acvfasymptotics} we obtain
\[ \lim_{h\to \infty} \frac{c(h)}{hr_h}=L.\]
Following the steps of the proof of Theorem \ref{theorem.mainresult} we obtain \eqref{eq:asymptotocs_c_discrete}.
\end{proof}
\section{Subexponential decay of the autocovariance function}
In this section we study the properties of the autocovariance function of the stationary solution of the main continuous-- and discrete--time equations \eqref{eq.stoch1} and \eqref{discr_stoch}  if the kernel $k$ is again regularly varying with index $-1-\alpha$, $\alpha>0$ but now  $a+\int_0^\infty k(s)\, d s<0$ or $a+\sum_{n=1}^\infty k_n<0$ holds respectively.

Then,  $k$ is a subexponential function or sequence in the sense of Appleby et al.~\cite{jaigdwr:2004b, App_exact_discrete}. In this case, the fundamental solution in both discrete-- (Theorem 3.2 in \cite{App_exact_discrete}) and continuous--time (Theorem 15 in \cite{jaigdwr:2004b}) decays at the same rate as the kernel $k$. Since $k$ is regularly varying with parameter $-\alpha-1<-1$, $r\in L^2([0, \infty);\mathbb{R})\cap L^1([0, \infty);\mathbb{R})$. This implies that the autocovariance function of the stationary solution is integrable. The next results show that nevertheless the autocovariance function  decays very slowly: it converges to zero at same rate as the kernel $k$, that is  polynomially.
\begin{remark}
If $a+\int_0^\infty k(s)\, ds>0$, then the fundamental solution grows exponentially: The characteristic function of $r$, a function $h$ which satisfies $\hat r(z)=1/h(z), \, \Re z\ge 0,$ is given by $h(z)=z-a-\hat k(z), \, z\in \mathbb{C},$ and satisfies $h(0)=-a-\int_0^\infty k(s)\, ds<0$.
Since $k$ is positive, we obtain for  $x>0$
\[ h(x)=x-a-\int_0^\infty e^{-xs} k(s)\, ds \ge x-a-\int_0^\infty k(s)\, ds,\]
which is positive if $x>a+\int_0^\infty k(s)\, ds>0$. Therefore, by the intermediate value theorem, there exists a positive root of the characteristic function. By the standard theory of Volterra equations this implies that the fundamental solution grows exponentially. Hence, the case $a+\int_0^\infty k(s)\, ds>0$ is not interesting for our research.
\end{remark}
\subsection{Continuous--time stochastic equation with subexponentially decaying memory}
Suppose $k\in C([0,\infty);(0,\infty))$ satisfies
\begin{enumerate}
\item[(S1)] $k\in \text{RV}_\infty(-1-\alpha)$ for $\alpha>0$,
\item[(S2)]$a+\int_0^\infty k(s) \, d s<0$.
\end{enumerate}
Theorem 15 in \cite{jaigdwr:2004b} yields, that the fundamental solution of \eqref{eq.fund1} converges to zero at the same rate as $k$:
\begin{equation} \label{r_asy_sub}
\lim_{t\to \infty} \frac{r(t)}{k(t)}=\frac{1}{\left(a+\int_0^\infty k(s)\, d s\right)^2}=:L_c.\end{equation}
Moreover, $r$ is also subexponential. Since $r$ is also square integrable, the stationary solution of \eqref{eq.stochstat} exists and the exact rate of decay of the autocovariance function can be determined.
\begin{theorem} \label{thm_suexp_cont}
Suppose $k$ satisfies (S1) and (S2). Let $r$ be solution of \eqref{eq.fund1}. Let $\sigma\in \mathbb{R}\setminus \{0\}$ and $B$ be the Brownian motion defined by \eqref{eq.whollineBM}. Then, the autocovariance function $c(\cdot)=\Cov(X(s), X(s+\cdot))$ of the stationary solution of \eqref{eq.stochstat} satisfies
\begin{equation}
\label{asy_s_sub}
\lim_{t\to \infty} \frac{c(t)}{k(t)}=\frac{\sigma^2}{\left(-a-\int_0^\infty k(s)\, d s\right)^3}>0.\end{equation}
\end{theorem}
\begin{proof}
The autocovariance function of the stationary solution is again given by \eqref{eq.statcovar}. Theorem 1.8.3 in \cite{Goldie} yields, that there exists a decaying  function $\lambda$ with $k(t)\sim \lambda(t)$ for $t\to \infty$. Since $r$ is integrable, we choose for an arbitrary $\epsilon>0$ a sufficiently large $T>0$, so that $2L_c\int_T^\infty |r(s)|\, d s<\epsilon.$
We now write
\begin{equation} \label{cov_split} \int_0^\infty \!\frac{r(t+s)r(s)}{k(t) }\, d s=\!\!\int_0^T\!\! \frac{r(t+s)}{k(t+s)}\frac{k(t+s)}{k(t) }r(s)\, d s+\int_T^\infty\!\frac{\lambda(t)}{k(t)} \frac{r(t+s)}{\lambda(t+s) }\frac{\lambda(t+s)}{\lambda(t)}r(s)\, d s. \end{equation}
The second integral is negligible:
since $\lambda$ is decreasing and $r(t)/\lambda(t)\to L_c$ for $t\to \infty$, the integrand is bounded for sufficiently large $t$ by $2L_c|r(s)|$. Hence
\[\limsup_{t\to \infty}\left|\int_T^\infty \frac{r(t+s)r(s)}{k(t) }\, ds \right|\le2 L_c\int_T^\infty |r(s)|\,d s <\epsilon.\]
Let us now consider the first integral in \eqref{cov_split}. With Potter's bound (cf.~\cite{Goldie}, Theorem 1.5.6) we obtain
\begin{equation} \label{potter_sub}\frac{k(t+s)}{k(t)}\to 1, \quad t\to \infty,\end{equation}
uniformly in $s$ for all $s<T$. Therefore for all sufficiently large $t$
\begin{equation}\label{potter1_sub} \sup_{s\le T} \left|\frac{k(t+s)}{k(t)}\right|\le 2\quad \mbox{and}\quad \sup_{s>0} \left|\frac{r(t+s)}{k(t+s)}\right|\le 2L_c.\end{equation}
Using dominated convergence theorem we obtain
\[ \lim_{t\to \infty}\int_0^T \frac{r(t+s)}{k(t+s)}\frac{k(t+s)}{k(t) }r(s)\, d s=L_c\int_0^T  r(s)\, d s.\]
Hence, the left--hand side of \eqref{asy_s_sub} converges to $L_c \int_0^\infty r(s)\, ds$ and the claim follows from the fact that $\int_0^\infty k(s)\, ds=-1/(a+\int_0^\infty k(s)\, ds)$.
\end{proof}
\begin{example} \label{ex2}Let $\alpha>0$ and
\begin{equation}
k(t)=\frac{1}{(1+t)^{\alpha+1}}, \quad t\geq0.
\end{equation}
We obtain the following convergence rate of the autocovariance function:
\[ \lim_{t\to \infty}c(t)t^{1+\alpha}=\frac{\sigma^2}{(-a-1/\alpha)^3}.\]
\end{example}
\begin{remark}
Examples \ref{ex1} and \ref{ex2} make clear that there is a very
different impact on the rate of convergence of the autocovariance
function from the decay rate of the kernel $k$ according as to
whether we are in the long--memory or subexponential case. In the
latter case, the rate of decay of the autocovariance function $c$ is
proportional to the rate of decay of the kernel $k$, so slow decay
in the memory as measured by the rate of decay of $k$ is reflected
exactly in the statistical memory, as measured by $c$. On the
contrary, in the long--memory case, a faster rate of decay of the
kernel $k$ results in a slower rate of decay of $c$. 
\end{remark}
\subsection{Discrete--time stochastic equation with subexponentially decaying memory}
Let us now consider the equation \eqref{discr_stoch} with a discrete  kernel $k=\{k_n:n\ge 1\}$ satisfying
\begin{enumerate}
\item[(S1')] $k$ is a regularly varying sequence with index $-1-\alpha$ for $\alpha>0$,
\item[(S2')] $a+\sum_{j=1}^\infty k_j <0.$ \end{enumerate}
Then $k$ satisfies the assumptions of the Theorem 3.2 in \cite{App_exact_discrete} and the fundamental solution of \eqref{eq.fund1} converges to zero at the same rate as $k$:
\begin{equation} \label{r_asy_sub_discr}
\lim_{n\to \infty} \frac{r_n}{k_n}=\frac{1}{\left(a+\sum_{j=1}^\infty k_j\right)^2}=:L_d.\end{equation}
Again, the stationary solution of \eqref{discr_stoch} exists and the exact rate of decay of the autocovariance function can be determined.
\begin{theorem}
Suppose $k$ satisfy (S1') and  (S2'). Let $r$ be solution of \eqref{discrete}. Let $\xi=\{\xi_n:n\in \mathbb{Z}\}$ be a  sequence of random variables defined as in Theorem \ref{theorem.discrete}.  Then, the autocovariance function $c(\cdot)=\Cov(X_n, X_{n+\cdot})$ of the stationary process defined in \eqref{eq.stochstat_discr} satisfies
\begin{equation}
\label{asy_s_sub_discr}
\lim_{h\to \infty} \frac{c(h)}{k_h}=\frac{\sigma^2}{\left(-a-\sum_{j=1}^\infty k_j\right)^3}>0.\end{equation}
\end{theorem}
\begin{proof}
The autocovariance function of the stationary solution is again given by \eqref{cov_discrete}. Since $(k_n)_{n\in \mathbb{N}}$ is a regularly varying sequence, the function $x\mapsto k(x):= k_{[x]}$ is a regularly varying function with index $-1-\alpha$. Hence, we may choose the function $\lambda$ as in the proof of Theorem  \ref{thm_suexp_cont}.  Since $r$ is absolutely summable, we choose for an  arbitrary $\epsilon>0$ a sufficiently large $N$, so that $2L_d\sum_{n=N+1}^\infty |r_n|<\epsilon.$ 
Similarly to the continuous case, we split the sum and study each term separately:
\begin{equation} \label{cov_split_discr} \sum_{n=0}^\infty  \frac{r_{n+h}r_n}{k_h }=\sum_{n=0}^N \frac{r_{n+h}}{k_{n+h}}\frac{k(n+h)}{k(h) }r_n+\sum_{n=N+1}^\infty \frac{\lambda(h)}{k(h)} \frac{r_{n+h}}{\lambda(n+h) }\frac{\lambda(n+h)}{\lambda(h)}r_n.\end{equation}
The sequence $r_h/\lambda(h)$ converges to $L_d$ as $h\to \infty$, so  the terms of the second sum are bounded for sufficiently large $h$ by $2L_d|r_n|$. Therefore,
\[\limsup_{h\to \infty}\left|\sum_{n=N+1}^\infty \frac{r_{n+h}r_n}{k_n }\right|\le2 L_d\sum_{n=N+1}^\infty| r_n|<\epsilon.\]
Let us now consider the first term in \eqref{cov_split_discr}. Applying Potter's bound to the function $k(x)$ as in \eqref{potter_sub} we obtain the discrete version of \eqref{potter1_sub}. Thus,
\[\lim_{h\to \infty} \sum_{n=0}^N  \frac{r_{n+h}r_n}{k_h }=L_d\sum_{n=0}^N r_n.\]
Similarly, $\sum_{n=0}^\infty r_n=-1/(a+\sum_{j=1}^\infty k_j)$ and the claim follows.
\end{proof}
\section{Proof of Theorem \ref{thm.existence covarfn stat}}\label{proof1}
First we show that the process defined by \eqref{eq.Xstat} has a continuous modification.
Applying  It\^{o}'s lemma, Cauchy-Schwarz inequality and Fubini's theorem we obtain\begin{align*}
\mathbb{E}((X(t)-X(u))^2)&=\sigma^2\int_{\mathbb{R}} (r(t-s)\I_{\{ s\le t\}}-r(u-s)\I_{\{ s\le u\}})^2\,d s\\
 &=\sigma^2\int_{\mathbb{R}} \left(\int_{u}^t r'(v-s) \,d v\right)^2 \,d s\\
 &\le \sigma^2 (t-u)\int_{\mathbb{R}} \int_{u}^t (r'(v-s))^2 \,d v \,d s\\
 &=\sigma^2(t-u)\int_{u}^t\int_{\mathbb{R}}(r'(s))^2 \,d s\,d v\\
 &\sigma^2(t-u)^2\int_{\mathbb{R}}(r'(s))^2 \,d s.
 \end{align*} Now, $r$ is square integrable and with $\| k \mathbin{*}r\|_{L^2}\le \|k\| _{L^1}\| r\|_{L^2}$, we have $r'\in L^2((0, \infty);\mathbb{R})$. Here, $(k \mathbin{*}r)(\cdot)$ denotes the convolution of $k$ and $r$, given by $\int_0^\cdot k(s)r(\cdot-s)\, ds$.
 The Kolmogorov-Chentsov theorem (see e.g~\cite{karatzas}, Theorem 2.8)  yields that $X$ has a continuous modification.
It remains to show that the process defined by  \eqref{eq.Xstat} solves \eqref{eq.stochstat}. We write
\begin{align*}
 \!\!X(t)-X(0)&=\int_{-\infty}^0 ((r(t-s)-r(-s))\,d B(s)+ \int_0^t r(t-s)\,d B(s)\\
&=\int_{-\infty}^t \int_0^t r'(u-s)\,d u \,d B(s) +\sigma B(t)\\
&=\int_{-\infty}^t \int_0^t\left( ar(u-s)+\int_0^{u-s}r(u-s-v)k(v)\,d v\right)\,d u\,d B(s)+\sigma B(t)\\
&= \int_0^t\!\!\int_{-\infty}^u\! \!\! \!ar(u\!-\!s)\, dB(s)\, du\!+\!\int_0^t\! \int_0^\infty \!\!\!\!\int_{-\infty}^{u-v}\!\!\! \!\!\!r(u\!-\!s\!-\!v)\,d B(s)\,d v\,d u+\sigma B(t)\\
&=\int_0^t aX(u)\, du +\int_0^t \int_0^\infty k(v)X(u-v)\,d v \,d u +\sigma B(t).
\end{align*}
Since $r$ and $B$ are continuous, we are able to apply stochastic Fubini' theorem (e.g. \cite{protter}, Ch.IV.6, Thm.~65), if
\[ \int_{-\infty}^t\int_0^t \!r(u-s)^2\, du \, ds<\infty \; \mbox{ and }\; \int_{-\infty}^t\int_0^t\left(\int_0^{u-s}\!\!\!\!r(u-s-v)k(v)\,d v\right)^2\, du \, ds<\infty.\]
The statement follows from classical Fubini's theorem and the fact that $r\in L^2([0, \infty);\mathbb{R})$ and $\| k \mathbin{*}r\|_{L^2}\le \|k\| _{L^1}\| r\|_{L^2}$.
\section{Poof of Theorem \ref{theorem.mainresult}} \label{proof}
Suppose that $r\in C([0,\infty);(0,\infty))$ obeys
\begin{equation} \label{eq.rinRV}
r\in \text{RV}_\infty(\mu) \quad\text{for some $\mu\in (-1,-1/2)$}.
\end{equation}
Since $r\in L^2([0, \infty);\mathbb{R})$, there exists $c:[0,\infty)\to (0,\infty)$ such that
\begin{equation} \label{def.c}
c(t) = \int_0^\infty r(s)r(s+t)\,ds, \quad t\geq 0.
\end{equation}
By assuming \eqref{eq.rinRV}, we exclude the possibility that $r\in L^1([0,\infty);\mathbb{R})$.
Our first result is the following rate of decay of $c$.
\begin{theorem} \label{theorem.acvfasymptotics}
Suppose that $r$ is a positive continuous function which
obeys  \eqref{eq.rinRV} for some $\mu\in (-1,-1/2)$.
Then the function $c$ in \eqref{def.c} is well--defined and moreover obeys
\begin{equation} \label{eq.casymptoticsr}
\lim_{t\to\infty} \frac{c(t)}{tr^2(t)} = \frac{\Gamma(-1-2\mu)\Gamma(1+\mu)}{\Gamma(-\mu)}=:L>0.
\end{equation}
\end{theorem}
\begin{proof}
For $\mu \in (-1,-1/2)$ we have
$\int_0^\infty  x^\mu (x+1)^\mu d x=L.$
First we suppose that $r$ is decreasing. In this case we choose  for an arbitrary $0<\epsilon<1$ a $\delta=\delta(\epsilon) >0$ such that
$\int_0^\delta  x^\mu (x+1)^\mu d x< \epsilon.$ The Uniform Convergence Theorem (\cite{Goldie}, Theorem 1.5.2) yields that
\[ \frac{r(tx)}{r(t)} \to x^\mu, \quad \mbox{uniformly in $x$, for all } x\ge \delta.\] Hence, there exists a $t_0=t_0(\delta)$ such that
\[ \frac{r(tx)r(t(x+1))}{r(t)^2}\le 2 x^{\mu} (x+1)^{\mu}, \quad \mbox{for all } t\ge t_0, x>\delta.\]
The function on the right hand side is integrable, hence, the dominated convergence theorem yields that
\[\lim_{t\to \infty} \int_\delta^\infty \frac{r(tx)r(t(x+1))}{r(t)^2}d x= \int_\delta^\infty\lim_{t\to \infty} \frac{r(tx)r(t(x+1))}{r(t)^2}d x= \int_\delta ^\infty x^\mu (x+1)^\mu d x.\]
There exists a $t_1=t_1(\delta)>t_0$ such that
\begin{gather}\left| L- \int_{\delta t}^\infty\frac{r(s)r(t+s)}{tr(t)^2}\,ds\right| = \left| L-\int_\delta^\infty \frac{r(tx)r(t(x+1))}{r(t)^2}d x\right|\notag\\
\le \epsilon +\left| \int_\delta ^\infty x^\mu (x+1)^\mu d x -\int_\delta^\infty \frac{r(tx)r(t(x+1))}{r(t)^2}d x\right| \le 2\epsilon \label{bound2}\end{gather}
for all $t\ge t_1$.
On the other hand, using the monotonicity of $r$ we obtain\[
 \int_0^{\delta t} \frac{r(s)r(t+s)}{tr(t)^2}\,ds\le  \int_0^{\delta t} \frac{r(s)}{tr(t)}\,ds=  \frac{R(t)}{r(t)t}\frac{R(\delta t)}{R(t) }, \]
where $R(t)=\int_0^t r(s)\,ds \in \mbox{RV}_\infty(\mu+1).$ It follows from Karamata's Theorem \cite{Goldie}, Theorem 1.5.11, that
\[ \frac{R(t)}{tr(t)} \to \frac{1}{\mu+1}.\] Choosing $\delta$ small enough and a $t_2(\delta)>t_1(\delta)$ large enough we obtain
\begin{equation} \int_0^{\delta t} \frac{r(s)r(t+s)}{tr(t)^2}\,ds \le 2 \lim_{t\to \infty}\frac{R(t)}{r(t)t}\frac{R(\delta t)}{R(t) }=2\frac{1}{\mu +1} \delta^{\mu +1} \le \epsilon,\label{bound1}\end{equation}
for all $t\ge t_2$.
Hence, combining (\ref{bound1}) and  (\ref{bound2}) we get for all $t\ge t_2$\[
\left| L-\int_{0}^\infty\frac{r(s)r(t+s)}{tr(t)^2}\,ds\right| \le \left| L-\int_{\delta t}^\infty\frac{r(s)r(t+s)}{tr(t)^2}\,ds\right|+\int_0^{\delta t}\frac{r(s)r(t+s)}{tr(t)^2}\,ds \le 3\epsilon.\]
Now, for arbitrary $r$ obeying \eqref{eq.rinRV}, let $\rho(t):= \sup\{ r(t)\::\: t\ge x\}$. Then $\rho$ is a positive decreasing function, continuous on $[0, \infty)$ and satisfying $\rho(x)\sim r(x)$ for $x \to \infty$ \cite{Goldie}, Theorem 1.5.3. For an arbitrary $\epsilon>0$ we choose $t_0=t_0(\epsilon)$ such that for all $t>t_0$ we have
\[\left| \frac{1}{ t\rho(t)^2}\int_0^\infty \rho(s)\rho(t+s)\, ds -L\right| \le \epsilon.\]
Since $r(t)/\rho(t)\to 1$ as $t\to\infty$ for every $\varepsilon\in(0,1)$ there exists
$t_1=t_1(\varepsilon)\ge t_0$ such that $1-\varepsilon<r(t)/\rho(t)<1+\varepsilon$ for all $t\geq t_1$.
Therefore
\begin{equation} \label{eq.lower}
 (1-\varepsilon)^2\le \frac{\int_{t_1}^\infty r(s)r(s+t)\,ds}{\int_{t_1}^\infty \rho(s)\rho(s+t)\,ds} \leq  (1+\varepsilon)^2 .
\end{equation}
For $\epsilon$ sufficiently small we obtain
\begin{equation} \label{eq.lower1}
 \left|\frac{\int_{t_1}^\infty r(s)r(s+t)\,ds}{\int_{t_1}^\infty \rho(s)\rho(s+t)\,ds} -1\right|\leq 3 \epsilon.
\end{equation}
Now, since $\rho$ is decreasing, we have for $t\ge t_1$
\begin{equation} \label{eq.finiteintegral nocontrib}
\frac{1}{\rho(t)}\int_0^{t_1} r(s)r(s+t)\,ds
= \int_0^{t_1} r(s)\frac{r(s+t)}{\rho(s+t)}\frac{\rho(s+t)}{\rho(t)}\,ds
\le (1+\epsilon)\int_0^{t_1} r(s)\,ds.
\end{equation}
Therefore as $t\mapsto t\rho(t)$ is in $\text{RV}_\infty(\mu+1)$ and $\mu+1>0$, we have $t\rho(t)\to \infty$
as $t\to\infty$, and so there exists a $t_2=t_2(\epsilon)\ge t_1$ such that
\[ \left|\frac{1}{t\rho^2(t)}\int_0^{t_1} r(s)r(s+t)\,ds\right| \le \epsilon \quad \mbox{and} \quad \left|\frac{1}{t\rho^2(t)}\int_0^{t_1} \rho(s)\rho(s+t)\,ds\right|\le \epsilon
\]
for all $t\ge t_2$.
Therefore,
\begin{gather*}
\left| \frac{1}{t\rho^2(t)}\int_0^{\infty} r(s)r(s+t)\,ds-L\right| \le \left| \frac{1}{t\rho^2(t)}\int_0^{\infty} \rho(s)\rho(s+t)\,ds-L\right|\\  +\left|\frac{1}{t\rho^2(t)}\int_0^{t_1} r(s)r(s+t)\,ds\right|
+\left|\frac{1}{t\rho^2(t)}\int_0^{t_1} \rho(s)\rho(s+t)\,ds\right| \\
+\left|\frac{1}{t\rho^2(t)}\int^\infty_{t_1} r(s)r(s+t)\,ds-\frac{1}{t\rho^2(t)}\int^\infty_{t_1} \rho(s)\rho(s+t)\,ds\right|\\
\le 3\epsilon +\frac{1}{t\rho^2(t)}\int^\infty_{t_1} \rho(s)\rho(s+t)\,ds \left|\frac{\int_{t_1}^\infty r(s)r(s+t)\,ds}{\int_{t_1}^\infty \rho(s)\rho(s+t)\,ds} -1\right|\\
\le 3\epsilon+ 3\epsilon L=(3+3L)\epsilon.
\end{gather*}
Finally we note that
\[\lim_{t\to \infty} \frac{c(t)}{r(t)^2}=\lim_{t\to \infty} \frac{c(t)}{\rho(t)^2}\frac{\rho(t)^2}{r(t)^2}=\lim_{t\to \infty} \frac{c(t)}{\rho(t)^2}.\]
\end{proof}

We now explicitly connect the result of Theorem~\ref{theorem.acvfasymptotics} to the autocovariance function
of the stationary solution of \eqref{eq.stochstat} in the case when $a=-\int_0^\infty k(s)\,ds$ to prove our main result.
\begin{proof}[Proof of Theorem \ref{theorem.mainresult}] $ $\\
It follows  from Theorem \ref{thm:explicitasymptotics} that
\begin{equation}
\lim_{t\to \infty} tr(t)\cdot  L^2(t)t^{1-2\alpha}=\frac{\sin^2 \alpha\pi}{\pi^2}.
\end{equation}
Since $\alpha\in (0,1/2)$ we have that $r\in L^2([0,\infty);(0,\infty))\cap C([0,\infty);(0,\infty))$ and $r\in \text{RV}_\infty(\alpha-1)$
with $\mu:=\alpha-1\in (-1,-1/2)$.
Therefore by Theorem~\ref{theorem.acvfasymptotics} and Theorem~\ref{thm.existence covarfn stat} obtain
\begin{align*}
\lim_{t\to\infty}  c(t) L^2(t)t^{1-2\alpha}
&=\lim_{t\to\infty} \frac{c(t)}{tr^2(t)}\cdot r^2(t) L^2(t)t^{2-2\alpha}\\
&= \sigma^2\frac{\Gamma(1-2\alpha)\Gamma(\alpha)}{\Gamma(1-\alpha)}\cdot \frac{ \sin^2(\pi \alpha)}{\pi^2},\\
\end{align*}
as claimed.
\end{proof}

\end{document}